\documentclass[runningheads]{llncs}

\usepackage{amsmath, amsfonts, amssymb, graphicx, xcolor, dsfont, enumerate, dcolumn,subcaption}
\usepackage[numbers,square, comma, compress]{natbib}
\usepackage{algorithm}
\usepackage{algorithmic}
\usepackage{makecell}
\usepackage{threeparttable}

\let\doendproof\endproof
\renewcommand\endproof{\hfill\qed\doendproof}

\usepackage{hyperref}
\usepackage{xurl}
\usepackage[capitalise]{cleveref}
\PassOptionsToPackage{hyphens}{url}\usepackage{hyperref}

\usepackage{tikz}

\newtheorem{thr}{Theorem}

\newtheorem{lem}{Lemma}
\newtheorem{prop}{Proposition}
\newtheorem{conj}{Conjecture}
\newtheorem{defi}{Definition}

\newtheorem{obs}{Observation}

\begin{document}

\title{Exhaustive generation of edge-girth-regular graphs}

%
\author{Jan Goedgebeur\inst{1,2}\orcidID{0000-0001-8984-2463} \and
Jorik Jooken\inst{1}\orcidID{0000-0002-5256-1921}}

\authorrunning{J. Goedgebeur and J. Jooken}

\institute{
Department of Computer Science, KU Leuven Kulak, 8500 Kortrijk, Belgium\\
\email{jan.goedgebeur@kuleuven.be,jorik.jooken@kuleuven.be}
\and
Department of Applied Mathematics, Computer Science and Statistics, Ghent University, 9000 Ghent, Belgium\\
}

\date{}

\maketitle

\begin{abstract}
Edge-girth-regular graphs (abbreviated as $egr$ graphs) are a class of highly regular graphs. More specifically, for integers $v$, $k$, $g$ and $\lambda$ an $egr(v,k,g,\lambda)$ graph is a $k$-regular graph with girth $g$ on $v$ vertices such that every edge is contained in exactly $\lambda$ cycles of length $g$. The central problem in this paper is determining $n(k,g,\lambda)$, which is defined as the smallest integer $v$ such that an $egr(v,k,g,\lambda)$ graph exists (or $\infty$ if no such graph exists) as well as determining the corresponding extremal graphs. We propose a linear time algorithm for computing how often an edge is contained in a cycle of length $g$, given a graph with girth $g$. We use this as one of the building blocks to propose another algorithm that can exhaustively generate all $egr(v,k,g,\lambda)$ graphs for fixed parameters $v, k, g$ and $\lambda$. We implement this algorithm and use it in a large-scale computation to obtain several new extremal graphs and improvements for lower and upper bounds from the literature for $n(k,g,\lambda)$. Among others, we show that $n(3,6,2)=24, n(3,8,8)=40, n(3,9,6)=60, n(3,9,8)=60, n(4,5,1)=30, n(4,6,9)=35, n(6,5,20)=42$ and we disprove a conjecture made by Araujo-Pardo and Leemans [\emph{Discrete Math.} \textbf{345}(10):112991 (2022)] for the cubic girth 8 and girth 12 cases. Based on our computations, we conjecture that $n(3,7,6)=n(3,8,10)=n(3,8,12)=n(3,8,14)=\infty.$ 

%
\end{abstract}


\section{Introduction}
In 2018, Jajcay, Kiss and Miklavi\v{c}~\cite{JKM18} introduced a new type of regularity called edge-girth-regularity. They define for integers $v$, $k$, $g$ and $\lambda$ an edge-girth-regular $(v,k,g,\lambda)$ graph (abbreviated as an $egr(v,k,g,\lambda)$ graph) as a $k$-regular graph with girth $g$ on $v$ vertices such that every edge is contained in exactly $\lambda$ cycles of length $g$. Edge-girth-regular graphs are related to a number of important graph classes. More precisely, edge-girth-regular graphs generalize the well-known notion of edge-regular graphs~\cite{BCN89} (in which each edge appears in the same number of triangles) and they are related to Moore graphs and the notoriously difficult \textit{cage problem}~\cite{EJ08}. This problem asks to determine the order of the smallest $k$-regular graph with girth $g$ as well as the corresponding extremal graphs, known as \textit{cages} (we refer the interested reader to~\cite{EJ08} for a more thorough overview of this topic). Many cages are known to be edge-girth-regular graphs. Moreover, edge-transitive graphs are an important subclass of edge-girth-regular graphs.

Since a disconnected graph $G$ can only be edge-girth-regular if each of its connected components are edge-girth-regular, all graphs in the remainder of this paper refer to connected graphs without loops and parallel edges. For the cage problem, it is known that for all integers $k,g \geq 3$ there exist infinitely many $k$-regular graphs of girth $g$~\cite{S63}. For edge-girth-regular graphs on the other hand, there can either be zero, finitely many or infinitely many edge-girth-regular graphs depending on the choice of the parameters $k, g$ and $\lambda$~\cite{JKM18}. Analogous to the important cage problem, Drglin, Filipovski, Jajcay and Raiman~\cite{DFJR21} consider the class of extremal edge-girth-regular graphs, which are edge-girth-regular graphs of minimum order for given parameters $k, g$ and $\lambda$. They initiated the study of $n(k,g,\lambda)$, which is defined as the smallest integer $v$ such that an $egr(v,k,g,\lambda)$ graph exists (or $\infty$ if no such graph exists) and determine $n(k,g,\lambda)$ exactly for a number of parameter triples $(k,g,\lambda)$ with a focus on the 3-regular and 4-regular case. However, the exact value of $n(k,g,\lambda)$ is only known for a handful of cases and most focus in the literature so far has been on lower and upper bounds instead. In the current paper, we determine $n(k,g,\lambda)$ exactly for a number of open cases (as well as the corresponding extremal graphs) and improve several existing lower and upper bounds from the literature~\cite{DFJR21,EJ08,JKM18,P23}. 

We now turn our attention to some basic properties of edge-girth-regular graphs that will be used throughout the paper. For fixed parameters $k$ and $g$, there are only finitely many $\lambda$ such that an $egr(v,k,g,\lambda)$ graph exists for some integer $v$. More precisely, Jajcay, Kiss and Miklavi\v{c}~\cite{JKM18} showed the following elemental properties.

\begin{prop}[\cite{JKM18}]\label{prop:elementalProperties}
    Let $G$ be an $egr(v,k,g,\lambda)$ graph. Then the following hold:
    \begin{enumerate}[(i)]
    \item every vertex $u$ of $G$ is contained in exactly $\frac{k\lambda}{2}$ cycles of length $g$;
    \item there are $\frac{vk\lambda}{2g}$ cycles of length $g$ in $G$;
    \item $k\lambda$ is an even integer;
    \item if $g$ is even, then $\lambda \leq (k-1)^{\frac{g}{2}}$;
  \item if $g$ is odd, then $\lambda \leq (k-1)^{\frac{g-1}{2}}.$
    \end{enumerate}
\end{prop}

A natural lower bound for the minimum order of any $k$-regular graph with finite girth $g$ is the Moore bound:
\[
    M(k,g) =
    \begin{cases}
        1+k+k(k-1)+\ldots+k(k-1)^{(g-3)/2}, & \text{$g$ odd,} \\
        2(1+(k-1)+\ldots+(k-1)^{(g-2)/2}), & \text{$g$ even.}
    \end{cases}
\]
Moore graphs are $k$-regular graphs with girth $g$ that attain this bound. We remark that all Moore graphs of even girth are edge-girth-regular~\cite{JFJ17}. In the case of edge-girth-regular graphs, Drglin, Filipovski, Jajcay and Raiman~\cite{DFJR21} gave the following improved lower bound, which will turn out to be sharp for a number of cases that we consider in the current paper.
\begin{thr}[Theorem 2.3 in \cite{DFJR21}]
Let $k$ and $g$ be a fixed pair of integers greater than or equal to 3, and let $\lambda \leq (k-1)^{\frac{g-1}{2}}$, when $g$ is odd, and $\lambda \leq (k-1)^{\frac{g}{2}}$, when $g$ is even. Then
\begin{align}\label{lowerBoundGOdd}
n(k,g,\lambda) \geq M(k,g)+(k-1)^{\frac{g-1}{2}}-\lambda\text{, for $g$ odd,}
\end{align}
and
\begin{align}\label{lowerBoundGEven}
n(k,g,\lambda) \geq M(k,g)+\left\lceil 2 \frac{(k-1)^{\frac{g}{2}}-\lambda}{k}\right\rceil\text{, for $g$ even.}
\end{align}
\end{thr}

The rest of this paper is structured as follows: in~\cref{subsec:not&def} we introduce notation and definitions that will be used throughout the paper. In~\cref{sec:ngcAlgo} we propose a linear time algorithm -- based on a recursion relation -- for computing the number of cycles of length $g$ in which an edge is contained, given a graph of girth $g$. Next, in~\cref{sec:genAlgo} we propose an algorithm to exhaustively generate all $egr(v,k,g,\lambda)$ graphs, which uses the algorithm from~\cref{sec:ngcAlgo} as an important ingredient. In~\cref{sec:bounds} we use these algorithms in a large-scale computation amounting to 6 CPU-years in order to improve several lower and upper bounds on $n(k,g,\lambda)$ and independently verify several claims using different algorithms. We also highlight several interesting graphs that we discovered and briefly discuss their properties. Among others, we discuss two graphs that disprove a conjecture by Araujo-Pardo and Leemans~\cite{AL22} for the cases of cubic graphs with girths 8 and 12. Finally, in~\cref{sec:conc} we conclude this paper by discussing possible avenues for further research and we conjecture that $n(3,7,6)=n(3,8,10)=n(3,8,12)=n(3,8,14)=\infty$ based on our computations.

\subsection{Notation and definitions}\label{subsec:not&def}

For an integer $n \geq 1$ we denote by $[n]$ the set $\{1,\ldots,n\}$. The \textit{girth} $g$ of a graph $G$ is the length of the shortest cycle of $G$ (or $\infty$ if $G$ is acyclic). A \textit{girth cycle} of $G$ is a cycle with length $g$. We define $ngc(G,u_1u_2)$ as the number of pairwise distinct girth cycles in $G$ containing the edge $u_1u_2$. An \textit{edge-girth-regular} $(v,k,g,\lambda)$ graph $G$ is a $k$-regular graph of girth $g$ and order $v$ such that for every edge $u_1u_2$, we have $ngc(G,u_1u_2)=\lambda$. Such a graph will be abbreviated as an $egr(v,k,g,\lambda)$ graph. We define $n(k,g,\lambda)$ as the smallest value of $v$ for which an $egr(v,k,g,\lambda)$ graph exists (and define it as $\infty$ if no such graph exists). We define $nsp(G,u_1,u_2)$ and $d(G,u_1,u_2)$ as, respectively, the number of pairwise distinct shortest paths and the distance in $G$ between vertices $u_1$ and $u_2$. For an edge $u_1u_2$ from the graph $G=(V,E)$, we use $G-u_1u_2$ to denote the graph $(V,E\setminus\{u_1u_2\})$. The set $N_G(u)$ denotes the neighbors of $u$ in the graph $G$ and $deg_G(u)=|N_G(u)|$ denotes the degree of $u$ in the graph $G$.

\section{An algorithm for calculating $ngc(G,u_1u_2)$}\label{sec:ngcAlgo}

Let $G=(V,E)$ be a graph and $u_1u_2$ be an edge of $G$. We are interested in obtaining an efficient algorithm for calculating $ngc(G,u_1u_2)$, assuming that the girth $g$ is known (this is indeed the case for the application that we have in mind in~\cref{sec:genAlgo}). The most straightforward idea would consist of trying to efficiently generate all girth cycles of $G$ which contain $u_1u_2$ and then counting. However, the number of girth cycles can be exponential in terms of $g$ so this cannot be done efficiently in general. Instead, in this section we show that $ngc(G,u_1u_2)$ can be calculated with a time complexity of $O(|V|+|E|)$ based on a recursion relation that does not involve enumerating girth cycles.

We first observe a relationship between girth cycles and shortest paths.
\begin{obs}
Let $u_1u_2$ be an edge of the graph $G=(V,E)$ with girth $g$. Now $ngc(G,u_1u_2)=0$ if $d(G-u_1u_2,u_1,u_2)>g-1$ and $ngc(G,u_1u_2)=nsp(G-u_1u_2,u_1,u_2)$ otherwise.
\end{obs}

Hence, we can instead focus on computing $d(G-u_1u_2,u_1,u_2)$ and $nsp(G-u_1u_2,u_1,u_2)$. Calculating the distance between two vertices is a classical problem that can be solved using breadth-first search and a small modification allows us to compute the number of shortest paths as well. This becomes clear in the following observation.

\begin{obs}
Let $D_i = \{u \in V(G)~|~d(G-u_1u_2,u,u_1)=i \}$ be the set of vertices at distance $i$ from $u_1$ in the graph $G-u_1u_2$. We now have the following base case (for $u=u_1$):
\begin{align*}
nsp(G-u_1u_2,u_1,u_1)=1.
\end{align*}
For a vertex $u \neq u_1$, let $d(G-u_1u_2,u,u_1)=i>0$. Since every shortest path of length $i$ can be obtained by adding an edge to a shortest path of length $i-1$, we have the following recursive case:
\begin{align*}
nsp(G-u_1u_2,u,u_1)=\sum_{u' \in N_{G-u_1u_2}(u) \cap D_{i-1}}nsp(G-u_1u_2,u',u_1).
\end{align*}
\end{obs}

These recursion relations can be used to calculate $ngc(G,u_1,u_2)$, assuming that the girth $g$ is known, with a time complexity of $O(|V|+|E|)$ using a breadth-first search algorithm (the pseudocode is given in~\cref{algo:calculateNGC}).

\begin{algorithm}[ht!]
\caption{calculateNumberGirthCycles(Graph $G$, Girth $g$, Edge $u_1u_2$)}
\label{algo:calculateNGC}
  \begin{algorithmic}[1]
            \STATE // This algorithm calculates $ngc(G,u_1u_2)$
            \STATE // Array with $|V|$ entries intialized to $\infty$; the array stores $d(G-u_1u_2,u,u_1)$
            \STATE $dist \gets newArray(|V|, \infty)$
            \STATE // Array with $|V|$ entries intialized to $0$; the array stores $nsp(G-u_1u_2,u,u_1)$
            \STATE $nsp \gets newArray(|V|,0)$
            \STATE $dist[u_1]=0$ // base case
            \STATE $nsp[u_1]=1$ // base case
            \STATE $q \gets emptyQueue()$
            \STATE $q.addLastElement(u_1)$
            \WHILE{$q$ is not empty}
            \STATE $currentVertex \gets q.firstElement()$
            \STATE $q.eraseFirstElement()$
            \STATE // $dist[u_2]$ and $nsp[u_2]$ cannot change anymore
            \IF{$currentVertex=u_2$}
                \IF{$dist[u_2]>g-1$}
                \RETURN 0
                \ELSE
                \RETURN $nsp[u_2]$
                \ENDIF
            \ENDIF
            \FOR{$u' \in N_{G-u_1u_2}(currentVertex)$}
                \IF{$dist[u']=\infty$}
                \STATE $dist[u'] \gets dist[currentVertex]+1$ // recursive case
                \STATE $q.addLastElement(u')$
                \ENDIF
                \IF{$dist[u']=dist[currentVertex]+1$}
                \STATE $nsp[u'] \gets nsp[u']+nsp[currentVertex]$ // recursive case
                \ENDIF
            \ENDFOR
            \ENDWHILE
            \RETURN 0  
  \end{algorithmic}
\end{algorithm}

\section{An algorithm for exhaustively generating all $egr(v,k,g,\lambda)$ graphs}\label{sec:genAlgo}
In this section, we describe an algorithm for efficiently generating all $egr(v,k,g,\lambda)$ graphs for given parameters $v$, $k$, $g$ and $\lambda.$ This algorithm shares similarities with the algorithms described in~\cite{MMN98} and~\cite{EMNN11} for generating cages. However, the current algorithm requires further structural insights related to edge-girth-regular graphs. We first need a lemma that tells us something about subgraphs of an $egr(v,k,g,\lambda)$ graph. The lemma and its proof essentially belong to~\cite[Th.~2.3]{DFJR21}, but we have slightly altered it to better highlight the results that we need.

\begin{lem}\label{lem:subgraphs}
Let $G$ be an $egr(v,k,g,\lambda)$ graph. For odd $g$, let $\mathcal{T}^{u_1}_{k,\frac{g-1}{2}}$ be a tree rooted at vertex $u_1$ in which every internal vertex has degree $k$ and every leaf is at distance $\frac{g-1}{2}$ from $u_1$. Let $v_1, v_2, \ldots, v_k$ be the children of $u_1$ and let $L_1, L_2, \ldots, L_k$ be the set of leaves of $\mathcal{T}^{u_1}_{k,\frac{g-1}{2}}$ at distance $\frac{g-3}{2}$ from $v_1, v_2, \ldots, v_k$, respectively. For even $g$, let $\mathcal{T}^{u_1,u_2}_{k,\frac{g}{2}-1}$ be a tree in which every internal vertex has degree $k$, consisting of the edge $u_1u_2$ and two disjoint trees rooted at respectively $u_1$ and $u_2$ such that the leaves of the two trees are at distance $\frac{g}{2}-1$ from $u_1$ and $u_2$, respectively. Let $L_1$ and $L_2$ be the set of leaves of $\mathcal{T}^{u_1,u_2}_{k,\frac{g}{2}-1}$ at distance $\frac{g}{2}-1$ from $u_1$ and $u_2$, respectively. Now the following hold:

\begin{enumerate}[(i)]
    \item If $g$ is odd (even), $\mathcal{T}^{u_1}_{k,\frac{g-1}{2}}$ ($\mathcal{T}^{u_1,u_2}_{k,\frac{g}{2}-1}$) occurs as a subgraph of $G$.
  \item If $g$ is odd (even), for every $i \in [k]$ ($i \in [2]$) there are exactly $\lambda$ edges in $G$ with one endpoint in $L_i$ and another endpoint in $\bigcup_{j}(L_j) \setminus L_i$.
\end{enumerate}
\end{lem}

We now need the following definition.
\begin{defi}
Given a graph $G=(V,E)$ on $v$ vertices and integers $k, g$ and $\lambda$, for vertices $u_1, u_2 \in V$, $u_1 \neq u_2$, we call the set $\{u_1,u_2\}$ a \textit{valid pair} (relative to $G$, $k$, $g$ and $\lambda$) if $u_1u_2 \notin E$, $deg_G(u_1)<k$, $deg_G(u_2)<k$, the girth of $(V,E \cup \{u_1u_2\})$ is at least $g$ and there is no edge in $(V,E \cup \{u_1,u_2\})$ which is contained in strictly more than $\lambda$ cycles of length $g$. For a vertex $u_1 \in V(G)$, we define the set $validPairs_G(u_1)$ as $\{\{u_1,u_2\}~|~ u_2 \in V(G)\text{ and }\{u_1,u_2\}\text{ is a valid pair}\}$.
\end{defi}

Based on Lemma~\ref{lem:subgraphs}, the algorithm starts from the tree $\mathcal{T}^{u_1}_{k,\frac{g-1}{2}}$ (or the tree $\mathcal{T}^{u_1,u_2}_{k,\frac{g}{2}-1}$ depending on the parity of $g$) and adds isolated vertices until the resulting graph has $v$ vertices. The algorithm then recursively adds edges to this graph to exhaustively generate all $egr(v,k,g,\lambda)$ graphs. More specifically, in each recursion step the algorithm branches by adding one edge in a number of different ways, ensuring that no $egr(v,k,g,\lambda)$ graphs are omitted from the search space. We observe that the algorithm only needs to consider edges corresponding to valid pairs.

\begin{obs}
Let $G=(V,E)$ be a graph on $v$ vertices and $u_1,u_2 \in V$ be vertices such that the graph $G'=(V,E \cup \{u_1u_2\})$ is a subgraph of some $egr(v,k,g,\lambda)$ graph for given integers $k, g$ and $\lambda$. Now $\{u_1,u_2\}$ is a valid pair.
\end{obs}
\begin{proof}
Since $G'$ is a subgraph an of $egr(v,k,g,\lambda)$ graph, we have $deg_{G'}(u_i) \leq k$ and thus $deg_G(u_i)<k$ ($i \in [2]$). Moreover, the girth of $G'$ cannot be strictly smaller than $g$, because adding edges to $G'$ cannot increase its girth.\ Similarly, there cannot be an edge in $G'$ which is contained in strictly more than $\lambda$ cycles of length $g$, because this number cannot decrease by adding more edges to $G'$.
\end{proof}

The algorithm adds edges in two phases. In the first phase, the algorithm only considers adding an edge between two vertices that are leaves of the tree $\mathcal{T}^{u_1}_{k,\frac{g-1}{2}}$ (or the tree $\mathcal{T}^{u_1,u_2}_{k,\frac{g}{2}-1}$). The first phase ends when condition (ii) of~\cref{lem:subgraphs} is met. In the second phase, the algorithm also allows all other types of edges to be added. The intuition behind these two phases is that the algorithm tries to maximally exploit known structural results about edge-girth-regular graphs in order to quickly detect which graphs can never occur as a subgraph of an edge-girth-regular graph (this follows the fail-first principle).

There are a number of different ways in which equivalent subgraphs can arise when edges are added. To overcome this inefficiency, the algorithm employs the following principles:
\begin{itemize}
    \item The order in which the edges are added does not matter (e.g.\ first adding edge $e_1$ and then adding edge $e_2$ is equivalent with first adding $e_2$ and then $e_1$). Therefore, the algorithm keeps track of which edges are eligible to be added at each recursive call. We note that the set of eligible edges is always a subset of the valid pairs.

	At each recursive call, the algorithm determines the vertex with degree strictly less than $k$ which has the least number of eligible edges that can be added (where ties are broken in an arbitrary fashion). The algorithm branches by iterating over all possibilities for the next edge that could be added incident with the current vertex. If the algorithm decides to add edge $e$ in the current node of the recursion tree, then it will mark $e$ as ineligible for all subsequent children in the recursion tree to avoid duplicate work.

	\item If $u_1, u_2$ and $u_3$ are pairwise distinct vertices and $u_2$ and $u_3$ are both isolated, then the graph $(V,E\cup\{u_1u_2\})$ is isomorphic to the graph $(V,E\cup\{u_1u_3\})$. The algorithm exploits this by keeping track of which vertices are isolated and only allows the addition of an edge between a non-isolated vertex and one particular isolated vertex.

	\item Isomorphic graphs may also arise in more complicated ways than what was described before. In general, the algorithm avoids this by computing a canonical form of the graph using the \textit{nauty} package~\cite{MP14}. Here, two graphs are isomorphic if and only if they have the same canonical form. The canonical forms are stored in a splay tree~\cite{ST85}, a classical data structure that allows efficient insertion and lookup of its elements. The algorithm uses this data structure to prune graphs for which an isomorphic graph was already previously constructed.
\end{itemize}

Finally, the algorithm also prunes graphs from the recursion tree for which there is some vertex such that there are not enough eligible edges that could be added in order to make the degree of that vertex equal to $k$, since such graphs can clearly never lead to $egr(v,k,g,\lambda)$ graphs. The pseudocodes of the function that recursively adds edges and the function that exhaustively generates all $egr(v,k,g,\lambda)$ graphs are shown in~\cref{algo:recAddEdges} and~\cref{algo:genAlgo}, respectively.

\begin{algorithm}[ht!]
\caption{recursivelyAddEdges(Graph $G=(V,E)$, Initial leaves $L$, Eligible edges $EE$, Integer $v$, Integer $k$, Integer $g$, Integer $\lambda$)}
\label{algo:recAddEdges}
  \begin{algorithmic}[1]
		\STATE // Each recursive call of this function adds one edge to the graph
		\STATE 
		\STATE // Do not do the same work twice
            \IF{function was called before with a graph as parameter that is isomorphic with $G$}
			\RETURN
            \ENDIF
		\STATE // No more edges need to be added
		\IF{$|E|=\frac{vk}{2}$}
			\IF{$G$ is an $egr(v,k,g,\lambda)$ graph}
				\STATE Output $G$
			\ENDIF
			\RETURN
            \ENDIF
		\STATE $verticesToConsider \gets \{u \in V~|~deg_G(u)<k\}$
		\IF{Condition (ii) of~\cref{lem:subgraphs} is not met}
			\STATE $verticesToConsider \gets verticesToConsider \cap L$
            \ENDIF
		\STATE // Choose $u_1$ as the vertex with the least number of adjacent eligible edges
		\STATE $u_1 \gets \arg\min_{u \in verticesToConsider}|EE[u]|$
		\STATE $edgesToConsider \gets EE[u_1]$	
		\IF{Condition (ii) of~\cref{lem:subgraphs} is not met}
			\STATE $edgesToConsider \gets edgesToConsider \cap (L \times L)$
            \ENDIF
		\STATE // Only keep at most one edge between a non-isolated vertex and an isolated vertex
		\STATE $edgesToConsider \gets removeAllButOnePendantEdge(edgesToConsider)$
		\STATE // Branch on the first edge incident with $u_1$ that is added
            \FOR{$u_1u_2 \in edgesToConsider$}
                \STATE $G' \gets (V,E\cup\{u_1u_2\})$
		   \STATE $EE' \gets updateEligibleEdges(G',EE)$
		   \IF{$deg_{G'}(u)+|EE'(u)| \geq k$ for all $u \in V$}
			\STATE $recursivelyAddEdges(G',L,EE',v,k,g,\lambda)$
		   \ENDIF
		   \STATE $EE \gets markAsIneligible(u_1u_2,EE)$
            \ENDFOR
  \end{algorithmic}
\end{algorithm}

\begin{algorithm}[ht!]
\caption{generateAllEdgeGirthRegularGraphs(Integer $v$, Integer $k$, Integer $g$, Integer $\lambda$)}
\label{algo:genAlgo}
  \begin{algorithmic}[1]
		\STATE // This function generates all $egr(v,k,g,\lambda)$ graphs
		\STATE $T \gets emptyTree()$
            \IF{$g$ is odd}
			\STATE $T \gets \mathcal{T}^{u_1}_{k,\frac{g-1}{2}}$
		\ELSE
			\STATE $T \gets \mathcal{T}^{u_1,u_2}_{k,\frac{g}{2}-1}$
            \ENDIF
		\STATE // Trivially, no $egr(v,k,g,\lambda)$ graph exists if $v$ is too small
		\IF{$v < |V(T)|$}
			\RETURN
            \ENDIF
		\STATE $L \gets leavesOf(T)$
		\STATE // Add isolated vertices until $G$ has $v$ vertices
		\STATE $G \gets (V(T) \cup [v-|V(T)|], E(T))$
		\STATE // $EE$ is a data structure indexed by vertices that stores a list of eligible edges for each vertex
		\STATE $EE \gets listOfLists(v)$
            \FOR{$u \in V(G)$}
                \STATE $EE[u] \gets validPairs_G(u)$
            \ENDFOR
		\STATE $recursivelyAddEdges(G,L,EE,v,k,g,\lambda)$
  \end{algorithmic}
\end{algorithm}

\subsection{Variants}\label{subsec:variants}
Additionally, we briefly describe two variants of the aforementioned algorithm. For most tuples $(v,k,g,\lambda)$ this algorithm is the best option, but for some tuples the variants are faster. The first variant is related to the two phases in which the edges are added. In some cases, it is faster to omit the first phase and directly go to the second phase in which all types of edges are allowed to be added. The second variant is related to the cost of computing whether a pair $\{u_1,u_2\}$ is valid for a graph $G=(V,E)$. More specifically, to determine if there is no edge in $G'=(V,E \cup \{u_1u_2\})$ that is contained in strictly more than $\lambda$ cycles of length $g$, the algorithm has to repeatedly calculate $ngc(G',e)$ for all edges $e \in E(G')$, which can be relatively expensive in practice. For some tuples $(v,k,g,\lambda)$ it is faster to assume that $\lambda$ is infinitely large for determining whether $\{u_1,u_2\}$ is a valid pair, because this allows us to avoid having to compute $ngc(G',e)$ (making this assumption maintains the correctness of the algorithm). In other words, the second variant also allows certain invalid edges to be added in the hope that this will quickly lead to graphs that can be pruned.

\section{Improved lower and upper bounds for $n(k,g,\lambda)$}\label{sec:bounds}

We implemented the algorithm from~\cref{sec:genAlgo} (as well as the variants discussed in~\cref{subsec:variants}). We executed these algorithms on a computer cluster for various tuples $(v,k,g,\lambda)$ to find lower and upper bounds for $n(k,g,\lambda)$ (as well as the corresponding graphs that attain the upper bounds). In total, these computations took around 6 CPU-years. If an algorithm was able to find an $egr(v,k,g,\lambda)$ graph, then clearly $n(k,g,\lambda) \leq v$. On the other hand, if an algorithm terminated for all integers $v' < v$ with input parameters $(v',k,g,\lambda)$ without generating any edge-girth-regular graphs, then clearly $n(k,g,\lambda) \geq v$ since the algorithm is exhaustive. For a given triple $(k,g,\lambda)$, the algorithm always started by generating all $egr(v,k,g,\lambda)$ graphs, where $v$ is the best available lower bound from the literature for $n(k,g,\lambda)$. The parameter $v$ was gradually increased if no edge-girth-regular graphs could be found, but skipping orders for which $vk$ is not even or $2g$ is not a divisor of $vk\lambda$ (see~\cref{prop:elementalProperties}). 

We also searched for edge-girth-regular graphs among available exhaustive lists of highly symmetrical graphs, which resulted in improved upper bounds for some cases that were too large for~\cref{algo:genAlgo} to find. More specifically, we tested all vertex-transitive graphs with order at most 47~\cite{HR20}, all graphs with order at most 26 having two vertex orbits~\cite{R19}, all 3-, 4-, 5- and 6-regular graphs available at House of Graphs~\cite{HOG}, all 3-regular vertex-transitive graphs with order at most 1280, all 4-regular arc-transitive graphs with order at most 640~\cite{PSV13} and all 5-regular arc-transitive graphs with order at most 500~\cite{P24}. We refer the interested reader to~\cref{sec:app1} for additional details about sanity checks and independent verifications related to the correctness of our implementations of the algorithms. All code and data related to this paper is made publicly available at \url{https://github.com/JorikJooken/edgeGirthRegularGraphs}. All graphs on at most 250 vertices are also made available at House of Graphs~\cite{HOG} by searching for the term “edge-girth-regular”.

So far, most attention in the literature has been given to the 3- and 4-regular case. In the current paper, we additionally focused on the 5- and 6-regular case. More precisely, we focused on the following pairs of $(k,g)$ that were computationally feasible for the exhaustive generation algorithm: $(3,3)$, $(3,4)$, $(3,5)$, $(3,6)$, $(3,7)$, $(3,8)$, $(4,3)$, $(4,4)$, $(4,5)$, $(4,6)$, $(5,3)$, $(5,4)$, $(5,5)$, $(6,3)$, $(6,4)$ and $(6,5)$. For the parameter $\lambda$, we considered all values for which the existence of an $egr(v,k,g,\lambda)$ graph was not ruled out by~\cref{prop:elementalProperties}. In Tables~\ref{tab:3RegBounds}-\ref{tab:6RegBounds} in~\cref{app:tables}, we summarized the bounds that we were able to obtain as well as the previous best available bounds from the literature (to the best of our knowledge) for the 3-, 4-, 5- and 6-regular case, respectively. Results where we were able to improve the best available lower or upper bound from the literature are shown in italics. We also marked results in bold where we were able to prove extremality (i.e.\ the lower bound equals the upper bound) and this was not already known from the literature. For the upper bounds, we indicated between brackets how many pairwise non-isomorphic graphs there are that attain this upper bound. For the lower bounds we additionally use the fact $vk$ is even and that $2g$ is a divisor of $vk\lambda$ (see~\cref{prop:elementalProperties}) without repeating this reference everywhere.

Moreover, we summarized the orders of the smallest $egr(v,3,g,\lambda)$ graphs among the vertex-transitive graphs (denoted as $n_{vt}(3,g,\lambda)$) for $9 \leq g \leq 16$ and the orders of the smallest $egr(v,4,g,\lambda)$ graphs among the arc-transitive graphs (denoted as $n_{at}(4,g,\lambda)$) for $7 \leq g \leq 10$ in Table~\ref{tab:highGirth} in~\cref{app:tables}. Note that these orders also yield upper bounds for $n(k,g,\lambda)$. Since the smallest 3-regular graphs with girth 9 have order 58~\cite{BMS95} and we determined that none of these graphs are edge-girth-regular, we also obtain $n(3,9,6)=60$ and $n(3,9,8)=60.$ Two graphs attaining this bound are shown in~\cref{fig:60Vertices}.

\begin{figure}[h!]
\begin{center}
\begin{tikzpicture}[scale=0.95]
  \def\sides{60}
  \def\radius{3}

  \foreach \i in {1,...,\sides} {
    \fill ({360/\sides * \i}:\radius) circle (2pt);
  }
  
  \foreach \i in {1,13,25,37,49} {
    \draw ({360/\sides * (\i + 1)}:\radius) -- ({360/\sides * \i}:\radius);
    \draw ({360/\sides * (\i + 51)}:\radius) -- ({360/\sides * \i}:\radius);
  }
  \foreach \i in {2,14,26,38,50} {
    \draw ({360/\sides * (\i + 1)}:\radius) -- ({360/\sides * \i}:\radius);
    \draw ({360/\sides * (\i + 18)}:\radius) -- ({360/\sides * \i}:\radius);
  }
  \foreach \i in {3,15,27,39,51} {
    \draw ({360/\sides * (\i + 1)}:\radius) -- ({360/\sides * \i}:\radius);
    \draw ({360/\sides * (\i + 42)}:\radius) -- ({360/\sides * \i}:\radius);
  }
  \foreach \i in {4,16,28,40,52} {
    \draw ({360/\sides * (\i + 1)}:\radius) -- ({360/\sides * \i}:\radius);
    \draw ({360/\sides * (\i + 9)}:\radius) -- ({360/\sides * \i}:\radius);
  }
  \foreach \i in {5,17,29,41,53} {
    \draw ({360/\sides * (\i + 1)}:\radius) -- ({360/\sides * \i}:\radius);
    \draw ({360/\sides * (\i + 31)}:\radius) -- ({360/\sides * \i}:\radius);
  }
  \foreach \i in {6,18,30,42,54} {
    \draw ({360/\sides * (\i + 1)}:\radius) -- ({360/\sides * \i}:\radius);
    \draw ({360/\sides * (\i + 17)}:\radius) -- ({360/\sides * \i}:\radius);
  }
  \foreach \i in {7,19,31,43,55} {
    \draw ({360/\sides * (\i + 1)}:\radius) -- ({360/\sides * \i}:\radius);
    \draw ({360/\sides * (\i + 51)}:\radius) -- ({360/\sides * \i}:\radius);
  }
  \foreach \i in {8,20,32,44,56} {
    \draw ({360/\sides * (\i + 1)}:\radius) -- ({360/\sides * \i}:\radius);
    \draw ({360/\sides * (\i + 42)}:\radius) -- ({360/\sides * \i}:\radius);
  }
  \foreach \i in {9,21,33,45,57} {
    \draw ({360/\sides * (\i + 1)}:\radius) -- ({360/\sides * \i}:\radius);
    \draw ({360/\sides * (\i + 18)}:\radius) -- ({360/\sides * \i}:\radius);
  }
  \foreach \i in {10,22,34,46,58} {
    \draw ({360/\sides * (\i + 1)}:\radius) -- ({360/\sides * \i}:\radius);
    \draw ({360/\sides * (\i + 9)}:\radius) -- ({360/\sides * \i}:\radius);
  }
  \foreach \i in {11,23,35,47,59} {
    \draw ({360/\sides * (\i + 1)}:\radius) -- ({360/\sides * \i}:\radius);
    \draw ({360/\sides * (\i + 43)}:\radius) -- ({360/\sides * \i}:\radius);
  }
  \foreach \i in {12,24,36,48,60} {
    \draw ({360/\sides * (\i + 1)}:\radius) -- ({360/\sides * \i}:\radius);
    \draw ({360/\sides * (\i + 29)}:\radius) -- ({360/\sides * \i}:\radius);
  }
\end{tikzpicture} \quad
\begin{tikzpicture}[scale=0.95]
  \def\sides{60}
  \def\radius{3}

  \foreach \i in {1,...,\sides} {
    \fill ({360/\sides * \i}:\radius) circle (2pt);
  }
  
  \foreach \i in {1,13,25,37,49} {
    \draw ({360/\sides * (\i + 1)}:\radius) -- ({360/\sides * \i}:\radius);
    \draw ({360/\sides * (\i + 45)}:\radius) -- ({360/\sides * \i}:\radius);
  }
  \foreach \i in {2,14,26,38,50} {
    \draw ({360/\sides * (\i + 1)}:\radius) -- ({360/\sides * \i}:\radius);
    \draw ({360/\sides * (\i + 15)}:\radius) -- ({360/\sides * \i}:\radius);
  }
  \foreach \i in {3,15,27,39,51} {
    \draw ({360/\sides * (\i + 1)}:\radius) -- ({360/\sides * \i}:\radius);
    \draw ({360/\sides * (\i + 21)}:\radius) -- ({360/\sides * \i}:\radius);
  }
  \foreach \i in {4,16,28,40,52} {
    \draw ({360/\sides * (\i + 1)}:\radius) -- ({360/\sides * \i}:\radius);
    \draw ({360/\sides * (\i + 26)}:\radius) -- ({360/\sides * \i}:\radius);
  }
  \foreach \i in {5,17,29,41,53} {
    \draw ({360/\sides * (\i + 1)}:\radius) -- ({360/\sides * \i}:\radius);
    \draw ({360/\sides * (\i + 45)}:\radius) -- ({360/\sides * \i}:\radius);
  }
  \foreach \i in {6,18,30,42,54} {
    \draw ({360/\sides * (\i + 1)}:\radius) -- ({360/\sides * \i}:\radius);
    \draw ({360/\sides * (\i + 34)}:\radius) -- ({360/\sides * \i}:\radius);
  }
  \foreach \i in {7,19,31,43,55} {
    \draw ({360/\sides * (\i + 1)}:\radius) -- ({360/\sides * \i}:\radius);
    \draw ({360/\sides * (\i + 49)}:\radius) -- ({360/\sides * \i}:\radius);
  }
  \foreach \i in {8,20,32,44,56} {
    \draw ({360/\sides * (\i + 1)}:\radius) -- ({360/\sides * \i}:\radius);
    \draw ({360/\sides * (\i + 11)}:\radius) -- ({360/\sides * \i}:\radius);
  }
  \foreach \i in {9,21,33,45,57} {
    \draw ({360/\sides * (\i + 1)}:\radius) -- ({360/\sides * \i}:\radius);
    \draw ({360/\sides * (\i + 26)}:\radius) -- ({360/\sides * \i}:\radius);
  }
  \foreach \i in {10,22,34,46,58} {
    \draw ({360/\sides * (\i + 1)}:\radius) -- ({360/\sides * \i}:\radius);
    \draw ({360/\sides * (\i + 15)}:\radius) -- ({360/\sides * \i}:\radius);
  }
  \foreach \i in {11,23,35,47,59} {
    \draw ({360/\sides * (\i + 1)}:\radius) -- ({360/\sides * \i}:\radius);
    \draw ({360/\sides * (\i + 34)}:\radius) -- ({360/\sides * \i}:\radius);
  }
  \foreach \i in {12,24,36,48,60} {
    \draw ({360/\sides * (\i + 1)}:\radius) -- ({360/\sides * \i}:\radius);
    \draw ({360/\sides * (\i + 39)}:\radius) -- ({360/\sides * \i}:\radius);
  }
\end{tikzpicture}
\end{center}
\caption{An extremal $egr(60,3,9,6)$ graph (left) and an extremal $egr(60,3,9,8)$ graph (right).}\label{fig:60Vertices} 
\end{figure}
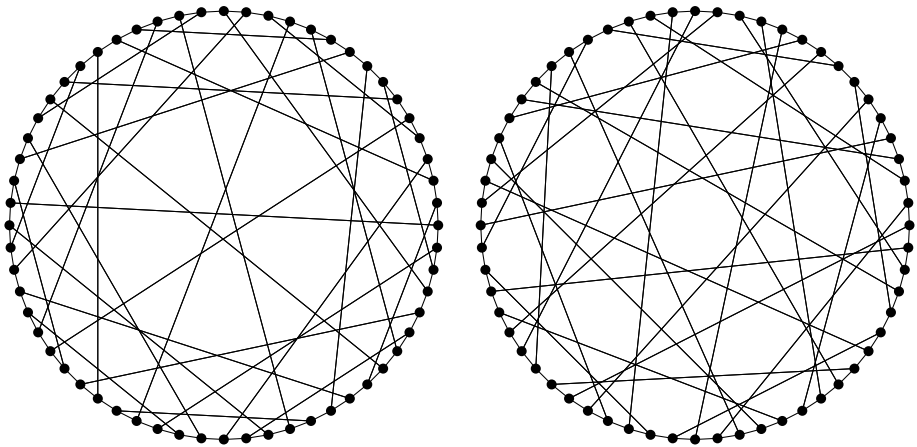

Apart from these two graphs, several new extremal graphs were determined and we improved various lower and upper bounds from the literature. We now discuss a few remarkable observations that follow from these calculations.

We have $n(3,6,6)=16$ and it is attained by the M\"{o}bius-Kantor graph (a vertex-transitive graph on 16 vertices). This graph can also be used to obtain other extremal edge-girth-regular graphs. More specifically, we have $n(4,5,6)=20$ and it is attained by a graph that can be obtained by adding four vertices to the M\"{o}bius-Kantor graph and adding the appropriate edges (see~\cref{fig:MoebiusKantorAndOffspring}). This makes the graph on 20 vertices one of the few known extremal edge-girth-regular graphs which are not vertex-transitive nor edge-transitive. It is less well-known than the M\"{o}bius-Kantor graph, but appears for example in~\cite{CGJ23} as a 4-regular graph containing many connected induced subgraphs.

\begin{figure}[h!]
\begin{center}
\begin{tikzpicture}[scale=0.95]
  \def\sides{16}
  \def\radius{3}

  \foreach \i in {1,...,\sides} {
    \fill ({360/\sides * \i}:\radius) circle (2pt);
  }
  
  \foreach \i in {1,3,5,7,9,11,13,15} {
    \draw ({360/\sides * (\i + 1)}:\radius) -- ({360/\sides * \i}:\radius);
    \draw ({360/\sides * (\i + 11)}:\radius) -- ({360/\sides * \i}:\radius);
  }
  \foreach \i in {2,4,6,8,10,12,14,16} {
    \draw ({360/\sides * (\i + 1)}:\radius) -- ({360/\sides * \i}:\radius);
    \draw ({360/\sides * (\i + 5)}:\radius) -- ({360/\sides * \i}:\radius);
  } 
\end{tikzpicture} \quad
\begin{tikzpicture}[scale=0.95]
  \def\sides{16}
  \def\radius{3}
  \def\rat{2.2}

  \foreach \i in {1,3,5,7,9,11,13,15} {
    \draw ({360/\sides * (\i + 1)}:\radius) -- ({360/\sides * \i}:\radius);
    \draw ({360/\sides * (\i + 11)}:\radius) -- ({360/\sides * \i}:\radius);
  }
  \foreach \i in {2,4,6,8,10,12,14,16} {
    \draw ({360/\sides * (\i + 1)}:\radius) -- ({360/\sides * \i}:\radius);
    \draw ({360/\sides * (\i + 5)}:\radius) -- ({360/\sides * \i}:\radius);
  }

  \draw ({360/4 * 1 + 180/\sides}:\radius/\rat)  -- ({360/\sides * (3)}:\radius);
  \draw ({360/4 * 1 + 180/\sides}:\radius/\rat)  -- ({360/\sides * (8)}:\radius);
  \draw ({360/4 * 1 + 180/\sides}:\radius/\rat)  -- ({360/\sides * (11)}:\radius);
  \draw ({360/4 * 1 + 180/\sides}:\radius/\rat)  -- ({360/\sides * (16)}:\radius);

  \draw ({360/4 * 2 + 180/\sides}:\radius/\rat)  -- ({360/\sides * (1)}:\radius);
  \draw ({360/4 * 2 + 180/\sides}:\radius/\rat)  -- ({360/\sides * (6)}:\radius);
  \draw ({360/4 * 2 + 180/\sides}:\radius/\rat)  -- ({360/\sides * (9)}:\radius);
  \draw ({360/4 * 2 + 180/\sides}:\radius/\rat)  -- ({360/\sides * (14)}:\radius);

  \draw ({360/4 * 3 + 180/\sides}:\radius/\rat)  -- ({360/\sides * (2)}:\radius);
  \draw ({360/4 * 3 + 180/\sides}:\radius/\rat)  -- ({360/\sides * (5)}:\radius);
  \draw ({360/4 * 3 + 180/\sides}:\radius/\rat)  -- ({360/\sides * (10)}:\radius);
  \draw ({360/4 * 3 + 180/\sides}:\radius/\rat)  -- ({360/\sides * (13)}:\radius);

  \draw ({360/4 * 4 + 180/\sides}:\radius/\rat)  -- ({360/\sides * (4)}:\radius);
  \draw ({360/4 * 4 + 180/\sides}:\radius/\rat)  -- ({360/\sides * (7)}:\radius);
  \draw ({360/4 * 4 + 180/\sides}:\radius/\rat)  -- ({360/\sides * (12)}:\radius);
  \draw ({360/4 * 4 + 180/\sides}:\radius/\rat)  -- ({360/\sides * (15)}:\radius);

    \foreach \i in {3,8,11,16} {
    \draw[fill,black] ({360/\sides * \i}:\radius) circle (2pt);
  }
  \foreach \i in {1,6,9,14} {
    \draw[fill,green] ({360/\sides * \i}:\radius) circle (2pt);
  }
  \foreach \i in {2,5,10,13} {
    \draw[fill,blue] ({360/\sides * \i}:\radius) circle (2pt);
  }
  \foreach \i in {4,7,12,15} {
    \draw[fill,red] ({360/\sides * \i}:\radius) circle (2pt);
  }
  
   \draw[fill,black] ({360/4 * 1 + 180/\sides}:\radius/\rat) circle (2pt);
   \draw[fill,green] ({360/4 * 2 + 180/\sides}:\radius/\rat) circle (2pt);
   \draw[fill,blue] ({360/4 * 3 + 180/\sides}:\radius/\rat) circle (2pt);
   \draw[fill,red] ({360/4 * 4 + 180/\sides}:\radius/\rat) circle (2pt);
   
\end{tikzpicture}
\end{center}
\caption{An $egr(16,3,6,6)$ graph (the M\"{o}bius-Kantor graph) and an $egr(20,4,5,6)$ graph. Both graphs are the unique extremal graphs for these parameters.}\label{fig:MoebiusKantorAndOffspring}
\end{figure}
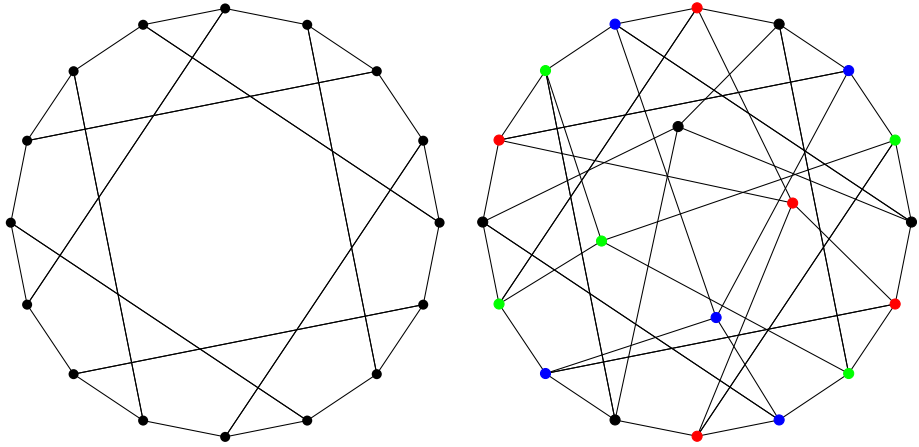

For the cage problem, an important question that has been open for more than 40 years asks whether every minimum order $k$-regular graph with even girth $g$ (i.e.\@ cage) is bipartite~\cite{EJ08,W82}. We remark that the analogous question for extremal edge-girth-regular graphs has a negative answer. For example,~\cref{fig:6416On15Vertices} shows the unique extremal $egr(15,6,4,16)$ that we determined in the current paper. This graph has girth 4 and is extremal, but is not bipartite.

\begin{figure}[h!]
\begin{center}
\begin{tikzpicture}
  \def\sides{15}
  \def\radius{3}

  \foreach \i in {1,...,\sides} {
    \fill ({360/\sides * \i}:\radius) circle (2pt);
  }
  
  \foreach \i in {1,...,\sides} {
    \draw ({360/\sides * (\i + 1)}:\radius) -- ({360/\sides * \i}:\radius);
    \draw ({360/\sides * (\i + 4)}:\radius) -- ({360/\sides * \i}:\radius);
    \draw ({360/\sides * (\i + 6)}:\radius) -- ({360/\sides * \i}:\radius);
    \draw ({360/\sides * (\i + 9)}:\radius) -- ({360/\sides * \i}:\radius);
    \draw ({360/\sides * (\i + 11)}:\radius) -- ({360/\sides * \i}:\radius);
  }
\end{tikzpicture}
\end{center}
\caption{The unique extremal $egr(15,6,4,16)$ graph is not bipartite and has even girth.}\label{fig:6416On15Vertices} 
\end{figure}
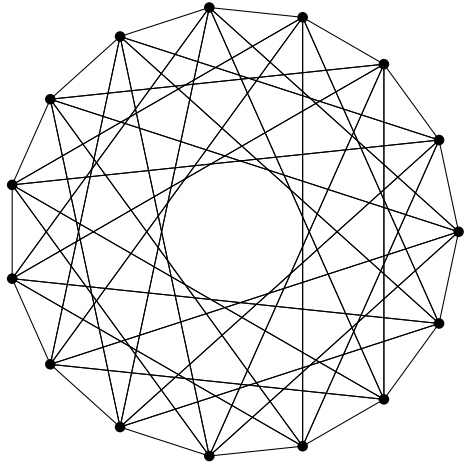

Porups{\'a}nszki describes the existence of an $egr(54,3,8,8)$ graph~\cite[Cor.\ 2.6]{P23}. Similarly, Yang, Sun and Zhang describe a Wenger graph, which is an $egr(54,3,8,8)$ graph~\cite[Th.\ 1]{YSZ23}. Araujo-Pardo and Leemans went one step further and made the following conjecture:
\begin{conj}[Conj. 4.4 in~\cite{AL22}]\label{conj:AL}
For $q \geq 3$ a prime power and $g \in \{8,12\}$ there exists a family of~$egr(2q^{\frac{g-2}{2}},q,g,(q-1)^{\frac{g-2}{2}}(q-2))$ graphs. These graphs are extremal edge-girth-regular graphs.
\end{conj}
However, a careful literature analysis reveals that Poto{\v{c}}nik and Vidali used the exhaustive list of all 3-regular vertex-transitive graphs until order 1280~\cite{PSV13} to filter out the edge-girth-regular graphs and they describe an $egr(v,3,8,8)$ graph in Table 2 of~\cite{PV22} without specifying its order $v$. In the current paper we prove that $n(3,8,8)=40$ and that this is attained by precisely one graph, which is vertex-transitive (shown on the left of~\cref{fig:451On30Vertices}). This disproves~\cref{conj:AL} for $q=3$ and $g=8$. We also showed that $n(3,12,32) \leq 162=2\cdot3^4$ (see Table~\ref{tab:highGirth} in~\cref{app:tables}), thereby also disproving~\cref{conj:AL} for $q=3$ and $g=12.$ Related to this, note that there does not exist any cubic connected vertex-transitive graph on $486=2\cdot3^5$ vertices (see~\cite{PSV13}).

We remark that there are 118 369 811 959 connected 4-regular graphs of order 28 with girth equal to 5 (see \url{https://oeis.org/A184945}), whereas this number is not known yet for order 30.
In the current paper, we show that $n(4,5,1)=30$ and it is attained by a unique graph (shown on the right of~\cref{fig:451On30Vertices}). This emphasizes that many of the results discussed in the current paper would have been computationally infeasible to obtain using existing algorithms.

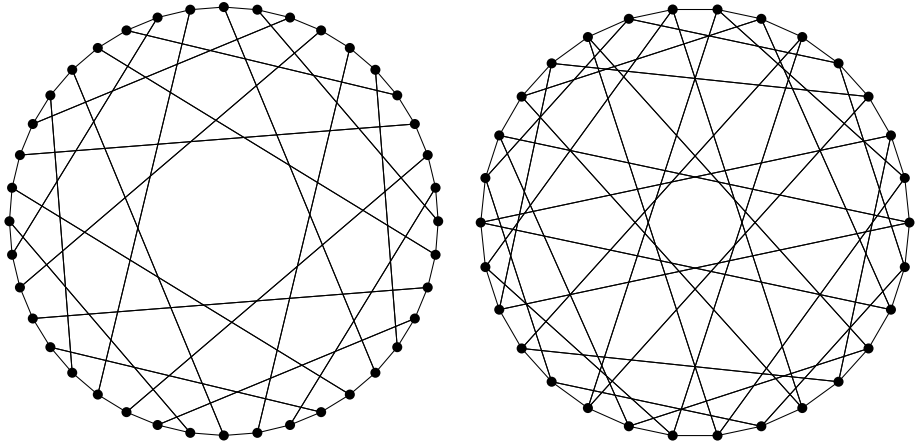
\begin{figure}[h!]
\begin{center}
\begin{tikzpicture}[scale=0.95]
  \def\sides{40}
  \def\radius{3}

  \foreach \i in {1,...,\sides} {
    \fill ({360/\sides * \i}:\radius) circle (2pt);
  }
  
  \foreach \i in {1,5,9,13,17,21,25,29,33,37} {
    \draw ({360/\sides * (\i + 1)}:\radius) -- ({360/\sides * \i}:\radius);
    \draw ({360/\sides * (\i + 31)}:\radius) -- ({360/\sides * \i}:\radius);
  }
   \foreach \i in {2,6,10,14,18,22,26,30,34,38} {
    \draw ({360/\sides * (\i + 1)}:\radius) -- ({360/\sides * \i}:\radius);
    \draw ({360/\sides * (\i + 25)}:\radius) -- ({360/\sides * \i}:\radius);
  }
   \foreach \i in {3,7,11,15,19,23,27,31,35,39} {
    \draw ({360/\sides * (\i + 1)}:\radius) -- ({360/\sides * \i}:\radius);
    \draw ({360/\sides * (\i + 15)}:\radius) -- ({360/\sides * \i}:\radius);
  }
   \foreach \i in {4,8,12,16,20,24,28,32,36,40} {
    \draw ({360/\sides * (\i + 1)}:\radius) -- ({360/\sides * \i}:\radius);
    \draw ({360/\sides * (\i + 9)}:\radius) -- ({360/\sides * \i}:\radius);
  }
  
\end{tikzpicture} \quad
\begin{tikzpicture}[scale=0.95]
  \def\sides{30}
  \def\radius{3}

  \foreach \i in {1,...,\sides} {
    \fill ({360/\sides * \i}:\radius) circle (2pt);
  }
  
  \foreach \i in {1,6,11,16,21,26} {
    \draw ({360/\sides * (\i + 1)}:\radius) -- ({360/\sides * \i}:\radius);
    \draw ({360/\sides * (\i + 6)}:\radius) -- ({360/\sides * \i}:\radius);
    \draw ({360/\sides * (\i + 22)}:\radius) -- ({360/\sides * \i}:\radius);
  }
  \foreach \i in {2,7,12,17,22,27} {
    \draw ({360/\sides * (\i + 1)}:\radius) -- ({360/\sides * \i}:\radius);
    \draw ({360/\sides * (\i + 13)}:\radius) -- ({360/\sides * \i}:\radius);
    \draw ({360/\sides * (\i + 24)}:\radius) -- ({360/\sides * \i}:\radius);
  }
  \foreach \i in {3,8,13,18,23,28} {
    \draw ({360/\sides * (\i + 1)}:\radius) -- ({360/\sides * \i}:\radius);
    \draw ({360/\sides * (\i + 8)}:\radius) -- ({360/\sides * \i}:\radius);
    \draw ({360/\sides * (\i + 17)}:\radius) -- ({360/\sides * \i}:\radius);
  }
  \foreach \i in {4,9,14,19,24,29} {
    \draw ({360/\sides * (\i + 1)}:\radius) -- ({360/\sides * \i}:\radius);
    \draw ({360/\sides * (\i + 5)}:\radius) -- ({360/\sides * \i}:\radius);
    \draw ({360/\sides * (\i + 25)}:\radius) -- ({360/\sides * \i}:\radius);
  }
  \foreach \i in {5,10,15,20,25,30} {
    \draw ({360/\sides * (\i + 1)}:\radius) -- ({360/\sides * \i}:\radius);
    \draw ({360/\sides * (\i + 13)}:\radius) -- ({360/\sides * \i}:\radius);
    \draw ({360/\sides * (\i + 17)}:\radius) -- ({360/\sides * \i}:\radius);
  }
  
\end{tikzpicture}
\end{center}
\caption{The unique extremal $egr(40,3,8,8)$ graph (left) and the unique extremal $egr(30,4,5,1)$ graph (right).}\label{fig:451On30Vertices} 
\end{figure}

\subsection{Independent verifications and sanity checks}\label{sec:app1}
Since the results of this paper rely on the outcome of algorithms, it is very important to take extra measures to ensure that the algorithms were implemented correctly. Therefore, we verified several claims independently using different algorithms and we also did several sanity checks that give us a lot of confidence that the algorithms were implemented correctly. The precise orders for which we were able to do this are different depending on the tuple $(k,g,\lambda)$, so in the interest of space we do not describe the precise orders. All results were in agreement for the following:
\begin{itemize}
    \item We implemented three different algorithms for the exhaustive generation of $egr(v,k,g,\lambda)$ graphs as discussed in~\cref{sec:genAlgo}. These three algorithms were always in agreement with each other, also for orders that are larger than $n(k,g,\lambda)$. For very large orders we were only able to use the most efficient variants due to the computational burden.
    \item We compared the outcome of these three algorithms with several known values of $n(k,g,\lambda)$ from the literature.
    \item We compared the outcome of these three algorithms with the outcome produced by generating all $k$-regular graphs of girth at least $g$ using the generator \textit{GENREG}~\cite{M99} followed by a filter that filters out the edge-girth-regular graphs.
    \item We changed the implementation of~\cref{algo:genAlgo} to generate all connected $k$-regular graphs of girth at least $g$. We compared the output of this algorithm with existing algorithms such as \textit{snarkhunter}~\cite{BGM11} and \textit{GENREG}~\cite{M99} for generating 3-, 4-, 5- and 6-regular graphs.
\end{itemize}

\section{Conclusion}\label{sec:conc}
In this paper, we developed an algorithm for exhaustively generating all $egr(v,k,g,\lambda)$ graphs and used this algorithm to improve existing lower and upper bounds on $n(k,g,\lambda)$. We believe that this algorithm could also be an important tool for future research to address several questions related to edge-girth-regular graphs. For example, there are several parameters $k$, $g$ and $\lambda$ for which it is unknown whether an $egr(v,k,g,\lambda)$ graph exists. For the cubic case, whenever $g \leq 8$ and some $egr(v,3,g,\lambda)$ graph is known to exist, there also exists a vertex-transitive $egr(v,3,g,\lambda)$ graph. Based on this observation and the computations that we did, we make the following conjecture:
\begin{conj}
The following holds: $n(3,7,6)=n(3,8,10)=n(3,8,12)=n(3,8,14)=\infty.$
\end{conj}
Moreover, in the literature several families of edge-girth-regular graphs are described, yielding upper bounds for $n(k,g,\lambda).$ We think it is worthwhile to investigate the relationships between different extremal edge-girth-regular graphs (e.g.\ large common subgraphs, systematic ways to transform one graph into the other and so on). We hope that a careful analysis of the new graphs that we discovered in the current paper can also lead to new constructions or infinite families and perhaps even shed more light on the cage problem.  Finally, we expect that our overview of the best existing bounds on $n(k,g,\lambda)$, which was largely missing from earlier literature, will also motivate other researchers to further reduce the gaps.

\section*{Acknowledgements}
We thank Tibo Van den Eede and Robert Jajcay for interesting discussions about edge-girth-regular graphs. Jan Goedgebeur is supported by Internal Funds of KU Leuven and an FWO grant with grant number G0AGX24N. Jorik Jooken is supported by an FWO grant with grant number 1222524N. The computational resources and services used in this work were provided by the VSC (Flemish Supercomputer Center), funded by the Research Foundation - Flanders (FWO) and the Flemish Government – department EWI.

\bibliographystyle{abbrv}
\bibliography{ref}

\newpage
\section*{Appendix}\label{sec: appendix}

\appendix
\section{Tables}\label{app:tables}
\begin{table}[h]
    \centering
    \begin{tabular}{|c| c | c | c  c | c c |}
\hline
\thead{$k$} & \thead{$g$} & \thead{$\lambda$} & \thead{Literature:\\$n(k,g,\lambda) \geq$} & \thead{Current paper:\\$n(k,g,\lambda) \geq$}  & \thead{Current paper:\\$n(k,g,\lambda) \leq$} & \thead{Literature:\\$n(k,g,\lambda) \leq$}\\
\hline
3&3&2&4~\cite[Th.\ 2.3]{DFJR21}&4&4 (1 graph) &4~\cite[Prop.\ 3.1]{JKM18}\\
\hline
3&4&2&8~\cite[Th.\ 2.3]{DFJR21}&8&8 (1 graph) &8~\cite[Prop.\ 3.2]{JKM18}\\
3&4&4&6~\cite[Th.\ 2.3]{DFJR21}&6&6 (1 graph) &6~\cite[Prop.\ 3.2]{JKM18}\\
\hline
3&5&2&20~\cite[Prop.\ 3.3]{JKM18}&20&20 (1 graph) &20~\cite[Prop.\ 3.3]{JKM18}\\
3&5&4&10~\cite{EJ08}&10&10 (1 graph) &10~\cite[Prop.\ 3.3]{JKM18}\\
\hline
3&6&2&18~\cite[Th.\ 2.3]{DFJR21}&\textbf{\textit{24}}&\textbf{24} (2 graphs) &\textbf{24}~\cite{DFJR21}\\
3&6&4&18~\cite[Th.\ 2.3]{DFJR21}&18&18 (1 graph) &18~\cite{JKM18}\\
3&6&6&16~\cite[Th.\ 2.3]{DFJR21}&16&16 (1 graph) &16~\cite{JKM18}\\
3&6&8&14~\cite{EJ08}&14&14 (1 graph) &14~\cite{EJ08}\\
\hline
3&7&2&28~\cite[Th.\ 2.3]{DFJR21}&\textit{42}&56 ($\geq$1 graph) &56~\cite{PV22}\\
3&7&$4$&28~\cite[Th.\ 2.3]{DFJR21}&28&28 (1 graph) &28~\cite{PV22}\\
3&7&6&28~\cite{EJ08}&\textit{42}&$\infty$&$\infty$\\
3&7&8&$\infty$~\cite[Prop.\ 2.5]{JKM18}&$\infty$&$\infty$&$\infty$\\
\hline
3&8&2&40~\cite[Th.\ 2.3]{DFJR21}&\textit{56}&64 ($\geq$1 graph) &64~\cite{PV22}\\
3&8&4&40~\cite[Th.\ 2.3]{DFJR21}&\textbf{\textit{48}}&\textbf{48} ($\geq$1 graph) &\textbf{48}~\cite{PV22}\\
3&8&6&40~\cite[Th.\ 2.3]{DFJR21}&\textbf{\textit{48}}&\textbf{48} ($\geq$1 graph) &\textbf{48}~\cite{PV22}\\
3&8&8&36~\cite[Th.\ 2.3]{DFJR21}&\textbf{\textit{40}}&\textbf{40} (1 graph) &\textbf{40}~\cite{PV22}\\
3&8&10&40~\cite[Th.\ 2.3]{DFJR21}&\textit{48}&$\infty$&$\infty$\\
3&8&12&36~\cite[Th.\ 2.3]{DFJR21}&\textit{48}&$\infty$&$\infty$\\
3&8&14&32~\cite[Th.\ 2.3]{DFJR21}&\textit{64}&$\infty$&$\infty$\\
3&8&16&30~\cite{EJ08}&30&30 (1 graph) &$30$~\cite{PV22}\\
\hline
\end{tabular}
    \caption{An overview of the best lower and upper bounds for $n(3,g,\lambda)$ for $g \leq 8.$}
    \label{tab:3RegBounds}
\end{table}

\begin{table}[h]
    \centering
    \begin{tabular}{|c| c | c | c  c | c c | }
\hline
\thead{$k$} & \thead{$g$} & \thead{$\lambda$} & \thead{Literature:\\$n(k,g,\lambda) \geq$} & \thead{Current paper:\\$n(k,g,\lambda) \geq$}  & \thead{Current paper:\\$n(k,g,\lambda) \leq$} & \thead{Literature:\\$n(k,g,\lambda) \leq$} \\
\hline
4&3&1&9~\cite[Prop.\ 3.1]{DFJR21}&9&9 (1 graph) &9~\cite{JKM18}\\
4&3&2&6~\cite[Th.\ 2.3]{DFJR21}&6&6 (1 graph) &6~\cite{JKM18}\\
4&3&3&5~\cite[Th.\ 2.3]{DFJR21}&5&5 (1 graph) &5~\cite{JKM18}\\
\hline
4&4&1&18~\cite{DFJR21}&18&18 (4 graphs) &18~\cite{DFJR21}\\
4&4&2&13~\cite{DFJR21}&13&13 (1 graph) &13~\cite{DFJR21}\\
4&4&3&14~\cite{DFJR21}&14&14 (1 graph) &14~\cite{DFJR21}\\
4&4&4&19~\cite{DFJR21}&\textit{21}&$\infty$&$\infty$\\
4&4&5&10~\cite[Th.\ 2.3]{DFJR21}&10&10 (1 graph) &10~\cite{DFJR21}\\
4&4&6&10~\cite[Th.\ 2.3]{DFJR21}&10&10 (1 graph) &10~\cite[Ex.\ 5.1]{KMS22}\\
4&4&$7 \leq \lambda \leq 8$&$\infty$~\cite[Th.\ 2.4]{KMS22}&$\infty$&$\infty$&$\infty$\\
4&4&9&8~\cite[Th.\ 2.3]{DFJR21}&8&8 (1 graph) &8~\cite{DFJR21}\\
\hline
4&5&1&25~\cite[Th.\ 2.3]{DFJR21}&\textbf{\textit{30}}&\textbf{\textit{30}} (1 graph) &13\,500~\cite[Th.\ 4.5]{JKM18}\\
4&5&2&25~\cite[Th.\ 2.3]{DFJR21}&\textbf{\textit{30}}&\textbf{\textit{30}} ($\geq$1 graph) &$\infty$\\
4&5&$3$&25~\cite[Th.\ 2.3]{DFJR21}&\textit{30}&\textit{55} ($\geq$1 graph)&$\infty$\\
4&5&$4$&25~\cite[Th.\ 2.3]{DFJR21}&\textit{30}&$\infty$&$\infty$\\
4&5&5&21~\cite[Th.\ 2.3]{DFJR21}&\textbf{\textit{24}}&\textbf{\textit{24}} (1 graph) &$\infty$\\
4&5&6&\textbf{20}~\cite[Th.\ 2.3]{DFJR21}&\textbf{20}&\textbf{\textit{20}} (1 graph) &$\infty$\\
4&5&7&20~\cite{EJ08}&\textit{35}&$\infty$&$\infty$\\
4&5&8&20~\cite{EJ08}&\textit{65}&$\infty$&$\infty$\\
4&5&9&$\infty$~\cite[Prop.\ 2.5]{JKM18}&$\infty$&$\infty$&$\infty$\\
\hline
4&6&1&39~\cite[Th.\ 2.3]{DFJR21}&\textit{57}&\textit{84} ($\geq$1 graph)&1\,658\,880~\cite[Th.\ 4.5]{JKM18}\\
4&6&2&39~\cite[Th.\ 2.3]{DFJR21}&\textit{51}&\textit{96} ($\geq$1 graph)&$\infty$\\
4&6&3&38~\cite[Th.\ 2.3]{DFJR21}&\textit{45}&90 ($\geq$1 graph)&60~\cite{D24}\\
4&6&4&39~\cite[Th.\ 2.3]{DFJR21}&\textit{45}&\textit{60} ($\geq$1 graph)&$\infty$\\
4&6&5&39~\cite[Th.\ 2.3]{DFJR21}&\textit{42}&\textit{81} ($\geq$1 graph)&$\infty$\\
4&6&6&37~\cite[Th.\ 2.3]{DFJR21}&\textit{40}&\textit{64} ($\geq$1 graph)&$\infty$\\
4&6&$7$&36~\cite[Th.\ 2.3]{DFJR21}&\textit{39}&\textit{60} ($\geq$1 graph)&$\infty$\\
4&6&$8$&36~\cite[Th.\ 2.3]{DFJR21}&\textit{39}&\textit{48} ($\geq$1 graph)&$\infty$\\
4&6&9&\textbf{35}~\cite[Th.\ 2.3]{DFJR21}&\textbf{35}&\textbf{\textit{35}} (1 graph) &$\infty$\\
4&6&10&36~\cite[Th.\ 2.3]{DFJR21}&\textit{39}&\textit{48} ($\geq$1 graph)&$\infty$\\
4&6&11&36~\cite[Th.\ 2.3]{DFJR21}&36&\textit{42} ($\geq$1 graph) &$\infty$\\
4&6&12&34~\cite[Th.\ 2.3]{DFJR21}&\textit{35}&\textit{40} ($\geq$1 graph) &$\infty$\\
4&6&$13$&33~\cite[Th.\ 2.3]{DFJR21}&\textit{36}&\textit{60} ($\geq$1 graph) &$\infty$\\
4&6&$14$&33~\cite[Th.\ 2.3]{DFJR21}&\textit{36}&$\infty$&$\infty$\\
4&6&15&32~\cite[Th.\ 2.3]{DFJR21}&\textit{33}&\textit{40} ($\geq$1 graph) &$\infty$\\
4&6&16&33~\cite[Th.\ 2.3]{DFJR21}&33&\textit{36} ($\geq$1 graph) &$\infty$\\
4&6&17&33~\cite[Th.\ 2.3]{DFJR21}&33&$\infty$&$\infty$\\
4&6&18&31~\cite[Th.\ 2.3]{DFJR21}&\textbf{\textit{32}}&\textbf{32} (1 graph)  &\textbf{32}~\cite[Ex.\ 5.6]{KMS22}\\
4&6&$19 \leq \lambda \leq 20$&30~\cite[Th.\ 2.3]{DFJR21}&\textit{36}&$\infty$&$\infty$\\
4&6&21&29~\cite[Th.\ 2.3]{DFJR21}&\textbf{\textit{30}}&\textbf{30} (1 graph) &\textbf{30}~\cite[Ex.\ 5.4]{KMS22}\\
4&6&22&30~\cite[Th.\ 2.3]{DFJR21}&\textit{39}&$\infty$&$\infty$\\
4&6&23&30~\cite[Th.\ 2.3]{DFJR21}&\textit{42}&$\infty$&$\infty$\\
4&6&24&28~\cite[Th.\ 2.3]{DFJR21}&28&28 (1 graph) &28~\cite[Ex.\ 5.2]{KMS22}\\
4&6&$25 \leq \lambda \leq 26$&$\infty$~\cite[Th.\ 2.4]{KMS22}&$\infty$&$\infty$&$\infty$\\
4&6&27&26~\cite{EJ08}&26&26 (1 graph) &26~\cite{EJ08}\\
\hline
\end{tabular}
    \caption{An overview of the best lower and upper bounds for $n(4,g,\lambda)$ for $g \leq 6.$}
    \label{tab:4RegBounds}
\end{table}

\begin{table}[h]
    \centering
    \begin{tabular}{|c| c | c | c  c | c c |}
\hline
\thead{$k$} & \thead{$g$} & \thead{$\lambda$} & \thead{Literature:\\$n(k,g,\lambda) \geq$} & \thead{Current paper:\\$n(k,g,\lambda) \geq$}  & \thead{Current paper:\\$n(k,g,\lambda) \leq$} & \thead{Literature:\\$n(k,g,\lambda) \leq$}\\
\hline
5&3&2&\textbf{12}~\cite[Lem. 1]{CJRS12}&\textbf{12}&\textbf{\textit{12}} (1 graph) &$\infty$\\
5&3&4&6~\cite[Th.\ 2.3]{DFJR21}&6&6 (1 graph) &6~\cite{YSZ23}\\
\hline
5&4&2&16~\cite[Th.\ 2.3]{DFJR21}&\textit{24}&\textit{32} ($\geq$1 graph) &$\infty$\\
5&4&4&\textbf{16}~\cite[Th.\ 2.3]{DFJR21}&\textbf{16}&\textbf{\textit{16}} (1 graph) &$\infty$\\
5&4&6&16~\cite[Th.\ 2.3]{DFJR21}&\textit{20}&$\infty$&$\infty$\\
5&4&8&14~\cite[Th.\ 2.3]{DFJR21}&\textit{20}&$\infty$&$\infty$\\
5&4&10&16~\cite[Th.\ 2.3]{DFJR21}&\textit{20}&$\infty$&$\infty$\\
5&4&12&12~\cite[Th.\ 2.3]{DFJR21}&12&12 (1 graph) &12~\cite[Ex.\ 5.1]{KMS22}\\
5&4&14&$\infty$~\cite[Th.\ 2.4]{KMS22}&$\infty$&$\infty$&$\infty$\\
5&4&16&10~\cite[Th.\ 2.3]{DFJR21}&10&10 (1 graph) &10~\cite[Cor.\ 3.6]{P23}\\
\hline
5&5&2&40~\cite[Th.\ 2.3]{DFJR21}&\textit{44}&\textit{60} ($\geq$1 graph) &60~\cite{D24}\\
5&5&4&38~\cite[Th.\ 2.3]{DFJR21}&38&\textit{66} ($\geq$1 graph) &$\infty$\\
5&5&6&36~\cite[Th.\ 2.3]{DFJR21}&36&$\infty$&$\infty$\\
5&5&8&34~\cite[Th.\ 2.3]{DFJR21}&34&\textit{36} ($\geq$1 graph) &$\infty$\\
5&5&10&32~\cite[Th.\ 2.3]{DFJR21}&32&$\infty$&$\infty$\\
5&5&12&32~\cite[Th.\ 4.4]{AL22}&32&32 ($\geq$1 graph) &32~\cite[Th.\ 2.2]{AL22}\\
5&5&14&30~\cite{EJ08}&\textit{42}&$\infty$&$\infty$\\
5&5&16&$\infty$~\cite[Prop.\ 2.5]{JKM18}&$\infty$&$\infty$&$\infty$\\
\hline
\end{tabular}
    \caption{An overview of the best lower and upper bounds for $n(5,g,\lambda)$ for $g \leq 5.$}
    \label{tab:5RegBounds}
\end{table}

\begin{table}[h]
    \centering
    \begin{tabular}{|c| c | c | c  c | c c |}
\hline
\thead{$k$} & \thead{$g$} & \thead{$\lambda$} & \thead{Literature:\\$n(k,g,\lambda) \geq$} & \thead{Current paper:\\$n(k,g,\lambda) \geq$}  & \thead{Current paper:\\$n(k,g,\lambda) \leq$} & \thead{Literature:\\$n(k,g,\lambda) \leq$}\\
\hline
6&3&1&\textbf{15}~\cite[Prop.\ 3.1]{DFJR21}&\textbf{15}&\textbf{\textit{15}} (1 graph) &$\infty$\\
6&3&2&\textbf{12}~\cite[Lem. 1]{CJRS12}&\textbf{12}&\textbf{\textit{12}} (1 graph) &16~\cite[Th.\ 4.1]{JKM18}\\
6&3&3&\textbf{9}~\cite[Th.\ 2.3]{DFJR21}&\textbf{9}&\textbf{\textit{9}} (1 graph) &$\infty$\\
6&3&4&\textbf{8}~\cite[Th.\ 2.3]{DFJR21}&\textbf{8}&\textbf{\textit{8}} (1 graph) &$\infty$\\
6&3&5&7~\cite[Th.\ 2.3]{DFJR21}&7&7 (1 graph) &7~\cite{YSZ23}\\
\hline
6&4&1&32~\cite[Prop.\ 3.2]{DFJR21}&32&\textit{40} ($\geq$1 graph) &$\infty$\\
6&4&2&24~\cite[Th.\ 3.5]{P23}&\textit{30}&\textit{36} ($\geq$1 graph) &320~\cite[Th.\ 5.3]{JKM18}\\
6&4&3&20~\cite[Th.\ 2.3]{DFJR21}&\textit{28}&\textit{32} ($\geq$1 graph) &$\infty$\\
6&4&4&19~\cite[Th.\ 2.3]{DFJR21}&\textit{25}&\textit{27} ($\geq$1 graph) &$\infty$\\
6&4&5&20~\cite[Th.\ 2.3]{DFJR21}&\textbf{\textit{24}}&\textbf{\textit{24}} ($\geq$1 graph) &$\infty$\\
6&4&6&20~\cite[Th.\ 2.3]{DFJR21}&\textbf{\textit{24}}&\textbf{\textit{24}} ($\geq$1 graph) &36~\cite[Th.\ 4.2]{JKM18}\\
6&4&7&20~\cite[Th.\ 2.3]{DFJR21}&\textit{24}&\textit{36} ($\geq$1 graph) &$\infty$\\
6&4&8&18~\cite[Th.\ 2.3]{DFJR21}&\textit{22}&\textit{26} ($\geq$1 graph) &$\infty$\\
6&4&9&\textbf{20}~\cite[Th.\ 2.3]{DFJR21}&\textbf{20}&\textit{\textbf{20}} (2 graphs) &$\infty$\\
6&4&10&18~\cite[Th.\ 2.3]{DFJR21}&\textbf{\textit{20}}&\textbf{\textit{20}} (5 graphs) &$\infty$\\
6&4&11&20~\cite[Th.\ 2.3]{DFJR21}&20&$\infty$&$\infty$\\
6&4&12&17~\cite[Th.\ 2.3]{DFJR21}&\textbf{\textit{20}}&\textbf{\textit{20}} ($\geq$1 graph) &$\infty$\\
6&4&13&16~\cite[Th.\ 2.3]{DFJR21}&\textbf{\textit{20}}&\textbf{20} ($\geq$1 graph)&\textbf{20}~\cite{D24}\\
6&4&14&16~\cite[Th.\ 2.3]{DFJR21}&\textbf{\textit{18}}&\textbf{\textit{18}} (1 graph) &$\infty$\\
6&4&15&16~\cite[Th.\ 2.3]{DFJR21}&\textit{20}&$\infty$&$\infty$\\
6&4&16&\textbf{15}~\cite[Th.\ 2.3]{DFJR21}&\textbf{15}&\textbf{\textit{15}} (1 graph) &$\infty$\\
6&4&17&\textbf{16}~\cite[Th.\ 2.3]{DFJR21}&\textbf{16}&\textbf{16} (1 graph) &$\infty$\\
6&4&18&16~\cite[Th.\ 2.3]{DFJR21}&\textit{22}&$\infty$&$\infty$\\
6&4&19&16~\cite[Th.\ 2.3]{DFJR21}&\textit{24}&$\infty$&$\infty$\\
6&4&20&14~\cite[Th.\ 2.3]{DFJR21}&14&14 (1 graph) &14~\cite[Ex.\ 5.1]{KMS22}\\
6&4&$21 \leq \lambda \leq 24$&$\infty$~\cite[Th.\ 2.4]{KMS22}&$\infty$&$\infty$&$\infty$\\
6&4&25&12~\cite[Th.\ 2.3]{DFJR21}&12&12 (1 graph) &12~\cite[Cor.\ 3.6]{P23}\\
\hline
6&5&1&65~\cite[Th.\ 2.3]{DFJR21}&65&$\infty$&$\infty$\\
6&5&2&60~\cite[Th.\ 2.3]{DFJR21}&60&$\infty$&7620~\cite[Th.\ 5.3]{JKM18}\\
6&5&3&60~\cite[Th.\ 2.3]{DFJR21}&60&$\infty$&110~\cite{D24}\\
6&5&4&60~\cite[Th.\ 2.3]{DFJR21}&60&$\infty$&910~\cite[Th.\ 5.3]{JKM18}\\
6&5&5&57~\cite[Th.\ 2.3]{DFJR21}&57&$\infty$&$\infty$\\
6&5&6&60~\cite[Th.\ 2.3]{DFJR21}&60&$\infty$&$\infty$\\
6&5&$7 \leq \lambda \leq 9$&55~\cite[Th.\ 2.3]{DFJR21}&55&$\infty$&$\infty$\\
6&5&10&52~\cite[Th.\ 2.3]{DFJR21}&52&$\infty$&$\infty$\\
6&5&11&55~\cite[Th.\ 2.3]{DFJR21}&55&$\infty$&$\infty$\\
6&5&$12 \leq \lambda \leq 14$&50~\cite[Th.\ 2.3]{DFJR21}&50&$\infty$&$\infty$\\
6&5&15&47~\cite[Th.\ 2.3]{DFJR21}&47&57 ($\geq$1 graph)&57~\cite{D24}\\
6&5&16&50~\cite[Th.\ 2.3]{DFJR21}&50&$\infty$&$\infty$\\
6&5&$17 \leq \lambda \leq 19$&45~\cite[Th.\ 2.3]{DFJR21}&45&$\infty$&$\infty$\\
6&5&20&\textbf{42}~\cite[Th.\ 2.3]{DFJR21}&\textbf{42}&\textbf{\textit{42}} ($\geq$1 graph) &$\infty$\\
6&5&21&45~\cite[Th.\ 2.3]{DFJR21}&45&$\infty$&$\infty$\\
6&5&22&40~\cite{EJ08}&40&40 (1 graph) &40~\cite{KMS22}\\
6&5&$23 \leq \lambda \leq 24$&40~\cite{EJ08}&\textit{45}&$\infty$&$\infty$\\
6&5&25&$\infty$~\cite[Prop.\ 2.5]{JKM18}&$\infty$&$\infty$&$\infty$\\
\hline
\end{tabular}
    \caption{An overview of the best lower and upper bounds for $n(6,g,\lambda)$ for $g \leq 5.$}
    \label{tab:6RegBounds}
\end{table}

\captionsetup[figure]{name=Table}
\begin{figure}[h!]
  \begin{subfigure}[b]{0.33\textwidth}
    \renewcommand{\arraystretch}{0.9}
    \centering
    \begin{tabular}{|c| c | c | c |}
\hline
\thead{$k$} & \thead{$g$} & \thead{$\lambda$} & \thead{Current paper:\\$n_{vt}(3,g,\lambda) =$}\\
\hline
3&9&2&108\\
3&9&4&408\\
3&9&6&60\\
3&9&8&60\\
\hline
3&10&2&160\\
3&10&4&120\\
3&10&10&112\\
3&10&12&110\\
3&10&16&90\\
3&10&20&80\\
\hline
3&11&2&1012\\
3&11&4&1012\\
3&11&8&506\\
\hline
3&12&2&384\\
3&12&4&384\\
3&12&6&512\\
3&12&8&272\\
3&12&10&240\\
3&12&12&256\\
3&12&14&256\\
3&12&16&234\\
3&12&18&256\\
3&12&20&234\\
3&12&22&216\\
3&12&24&204\\
3&12&26&192\\
3&12&28&182\\
3&12&32&162\\
3&12&34&168\\
3&12&36&162\\
\hline
3&13&26&384\\
\hline
3&14&2&1092\\
3&14&6&1092\\
3&14&10&1008\\
3&14&14&768\\
3&14&16&1008\\
3&14&28&504\\
3&14&42&512\\
3&14&44&406\\
3&14&56&506\\
\hline
3&15&10&1248\\
3&15&12&1280\\
3&15&20&864\\
3&15&30&816\\
3&15&32&620\\
\hline
3&16&32&1250\\
3&16&72&1240\\
3&16&80&1012\\
3&16&96&1008\\
3&16&104&1280\\
\hline
\end{tabular}
    \caption{}
    \label{tab:3RegVTBounds}
  \end{subfigure}%
  \begin{subfigure}[b]{0.33\textwidth}
            \renewcommand{\arraystretch}{0.9}
    \centering
    \begin{tabular}{|c| c | c | c |}
\hline
\thead{$k$} & \thead{$g$} & \thead{$\lambda$} & \thead{Current paper:\\$n_{at}(4,g,\lambda) =$}\\
\hline
4&7&1&252\\
4&7&2&224\\
4&7&3&273\\
4&7&4&224\\
4&7&7&160\\
4&7&9&91\\
4&7&14&80\\
\hline
4&8&1&612\\
4&8&2&320\\
4&8&3&336\\
4&8&4&252\\
4&8&5&384\\
4&8&6&288\\
4&8&7&512\\
4&8&8&243\\
4&8&9&320\\
4&8&10&240\\
4&8&11&336\\
4&8&12&288\\
4&8&13&256\\
4&8&14&256\\
4&8&15&204\\
4&8&16&240\\
4&8&17&224\\
4&8&18&256\\
4&8&19&240\\
4&8&20&162\\
4&8&21&192\\
4&8&22&256\\
4&8&23&256\\
4&8&24&126\\
4&8&25&192\\
4&8&26&160\\
4&8&27&168\\
4&8&28&144\\
4&8&29&192\\
4&8&30&160\\
4&8&33&160\\
4&8&36&140\\
4&8&37&128\\
4&8&39&128\\
4&8&40&135\\
4&8&41&128\\
4&8&43&128\\
4&8&45&128\\
4&8&48&110\\
4&8&60&100\\
4&8&65&96\\
\hline
\end{tabular}
    \caption{}
    \label{tab:4RegATBoundsPart1}
  \end{subfigure}%
  \begin{subfigure}[b]{0.33\textwidth}
    \renewcommand{\arraystretch}{0.9}
    \centering
    \begin{tabular}{|c| c | c | c |}
\hline
\thead{$k$} & \thead{$g$} & \thead{$\lambda$} & \thead{Current paper:\\$n_{at}(4,g,\lambda) =$}\\
\hline
4&9&6&504\\
4&9&9&600\\
4&9&15&546\\
4&9&18&506\\
4&9&19&504\\
4&9&21&285\\
4&9&27&320\\
4&9&36&320\\
\hline
4&10&50&625\\
4&10&60&546\\
4&10&61&640\\
4&10&65&640\\
4&10&67&640\\
4&10&70&624\\
4&10&71&640\\
4&10&80&546\\
4&10&90&512\\
4&10&95&576\\
4&10&108&420\\
4&10&110&432\\
\hline
\end{tabular}
    \caption{}
    \label{tab:4RegATBoundsPart2}
  \end{subfigure}
  \caption{An overview of $n_{vt}(3,g,\lambda)$ for $9 \leq g \leq 16$ and $n_{at}(4,g,\lambda)$ for $7 \leq g \leq 10.$}
\label{tab:highGirth}
\end{figure}
\captionsetup[figure]{name=Figure}

\end{document}